\documentclass[10pt]{article}
\usepackage{nicefrac} 
\usepackage{fullpage} 
\usepackage{color}   
\usepackage{hyperref}
\hypersetup{
     colorlinks = true,
     linkcolor = blue,
     anchorcolor = blue,
     citecolor = blue,
     filecolor = blue,
     urlcolor = blue
     }
\usepackage{authblk}
\usepackage{amsmath}
\usepackage{amsthm}
\usepackage{amssymb}
\usepackage{amsfonts} 
\usepackage{mathtools}
\usepackage{graphicx}
\usepackage{xcolor} 
\usepackage{framed} 
\usepackage{ulem} 
\usepackage{collectbox}
\colorlet{shadecolor}{gray!15}
\usepackage{algorithm} 
\usepackage{algorithmic} 
\usepackage{mdframed}
\mdfdefinestyle{shade}{
backgroundcolor=black!10}
\mdfdefinestyle{box}{skipabove=0, leftmargin=0,%
    rightmargin=0,%
    innertopmargin=-2mm,%
    innerbottommargin=2mm,
    innerleftmargin=2,
    innerrightmargin=2}


\numberwithin{equation}{section}
\theoremstyle{plain}

\newtheorem{proposition}{Proposition} 
\newtheorem{corollary}{Corollary} 
\theoremstyle{definition}
\newtheorem{definition}{Definition}

\usepackage{thmtools}  

\declaretheorem[shaded={rulecolor=black, rulewidth=0.5pt, bgcolor=white},name=Lemma]{lembox}

\declaretheorem[shaded={rulecolor=black, rulewidth=0.5pt, bgcolor=white},name=Theorem]{thmbox}

\usepackage{array}
\usepackage{epstopdf}
\usepackage{hyperref}

\newcommand{\ex}{\mathbb{E}}

\newcommand{\inp}[2]{\left\langle#1, #2 \right\rangle}
\newcommand{\norm}[1]{\left\lVert#1\right\rVert}

\newcommand{\Om}[1]{\Omega\left( #1\right)}
\newcommand{\OO}[1]{O\left( #1\right)}
\newcommand{\oo}[1]{o\left( #1\right)}
\newcommand{\re}{\mathbb{R}}

\newcommand{\eps}{\epsilon}

\newcommand{\xs}{x^*}
\newcommand{\g}{\nabla f}

\newcommand{\ee}{\eta}
\newcommand{\mm}{\mu} 
\newcommand{\ccc}{A}

\newcommand{\co}{A_1}
\newcommand{\cbo}{A_2}
\newcommand{\cbt}{A_3}
\newcommand{\ini}{k_0}
\newcommand{\ab}{\xi}

\begin{document}
	\title{On Tight Convergence Rates of Without-replacement SGD}
	 \author{Kwangjun Ahn} 
 	\author{Suvrit Sra}
 \affil{\small{Department of Electrical Engineering and Computer Science\\ Massachusetts Institute of Technology \\ \texttt{\{kjahn,suvrit\}@mit.edu}}}
 
	\maketitle 
	\begin{abstract}
	  For solving finite-sum optimization problems, SGD without replacement sampling is empirically shown to outperform SGD. 
	  Denoting by  $n$ the number of components in the  cost and $K$ the number of epochs of the algorithm, several recent works have shown convergence rates  of without-replacement SGD that have better dependency on $n$ and $K$ than the baseline rate of $O(1/(nK))$ for SGD.
	  However, there are two main limitations shared among those works: the rates have extra poly-logarithmic factors on $nK$, and denoting by $\kappa$  the condition number of the problem, the rates hold after  $\kappa^c\log(nK)$ epochs for some $c>0$.
	 In this work, we overcome these limitations by analyzing  step sizes that vary across epochs.
	\end{abstract}
\section{Introduction} \label{intro}
Stochastic gradient descent (SGD) is a popular optimization method  for  cost functions of the form:
\begin{align} \label{def:F}
    F(x):=\frac{1}{n}\sum_{i=1}^n f_i(x)~~\text{for $x\in\re^d$}\,.
\end{align}
More specifically, instead of computing the full gradient $\nabla F(x)$ at each iteration, SGD computes a cheaper approximation of it by computing $\nabla f_i(x)$  where index $i$ is randomly sampled from $\{1,2,\dots,n\}$.
In sampling such an index at each iteration, one can naturally think of the following two options: (i) with-replacement and (ii) without-replacement samplings.
Interestingly, without-replacement SGD (also known as random reshuffle or random shuffle) has been empirically shown to outperform with-replacement SGD~\cite{bottou2009curiously,bottou2012stochastic}.
However, the traditional analysis of SGD only covers the with-replacement version, raising the question whether one can capture this phenomenon \emph{theoretically}.

The  first contribution is established by G{\"u}rb{\"u}zbalaban, Ozdaglar, and Parrilo~\cite{gurbuzbalaban2015random} for the case where $F$ is $\mu$-strongly convex and each $f_i$ is quadratic and $L$-smooth.
Under this setting, they prove an asymptotic\footnote{Their proof is based on an asymptotic version of Chung's Lemma~\cite[Lemma 4.2]{fabian1967stochastic}.} convergence rate of $\OO{\nicefrac{1}{K^2}}$ for $K$ epochs (the number of entire passes through the component indices $\{1,2,\dots ,n\}$) when $n$ is treated as a constant.
This is indeed an asymptotic improvement over the convergence rate of $\OO{\nicefrac{1}{nK}}$ achieved by  with-replacement SGD.
However, in modern applications, $n$ cannot be treated as a constant, and it is important to characterize \emph{non-asymptotic} behavior of the algorithm. 
For instance, in machine learning applications, $n$ is typically equal to the number of data in a training set, and the algorithm is typically run a few epochs and terminates.  
Hence, it is  important to investigate the dependence on $n$ as well as non-asymptotic convergence rates.

Scrutinizing the analysis of~\cite[(58)]{gurbuzbalaban2015random}, one can actually deduce the asymptotic convergence rate of $\OO{\nicefrac{1}{(nK)^2}}+o(\nicefrac{1}{K^2})$, yet the exact dependence on $n$ and making this analysis non-asymptotic\footnote{Note that the key ingredient for making the analysis in \cite{gurbuzbalaban2015random} non-asymptotic is a non-asymptotic version of Chung's lemma~\cite[Lemma 1]{chung1954stochastic}, but as pointed out by Fabian~\cite[Discussion above Lemma 4.2]{fabian1967stochastic} the original proof of the result has some errors.} remain open.

\begin{table}
\centering
\begin{tabular}{ |c|c|c|  }
\hline
 \multicolumn{3}{|c|}{Sum of quadratics} \\
 \hline 
  \hline  
 Upper bound in~\cite{gurbuzbalaban2015random}&  $\OO{\frac{1}{(nK)^2}} +\oo{\frac{1}{K^2}}$ &asymptotic result
 \\
 Upper bound in~\cite{haochen2018random} & $\OO{\frac{\log^3(nK)}{(nK)^2} + \frac{\log^4(nK)}{K^3}}$& requires $K\geq \kappa \log (nK)$ \\
  Upper bound in~\cite{rajput2020closing} & $\OO{\frac{\log^2(nK)}{(nK)^2} + \frac{\log^3(nK)}{nK^3}}$& requires $K\geq \kappa^2 \log (nK)$\\
   Lower bound in~\cite{safran2019good} & $\Om{\frac{1}{(nK)^2} + \frac{1}{nK^3}}$&  constant step size\\
  \hline 
  Our upper bound & $\OO{\frac{1}{(nK)^2} + \frac{1}{nK^3}}$& all   $K\geq 1$\\
 \hline
\end{tabular}

\vspace*{0.3 cm}

\begin{tabular}{ |c|c|c|}
\hline
 \multicolumn{3}{|c|}{Non-quadratic strongly convex} \\
 \hline 
  \hline  
 Upper bound in~\cite{nagaraj2019sgd}&$\OO{\frac{\log^2(nK)}{nK^2}}$&requires $K\geq \kappa^2 \log (nK)$  
 \\
 Lower bound in~\cite{rajput2020closing} & $\Om{\frac{1}{nK^2}}$ & constant step size\\
 \hline
  Our upper bound & $\OO{\frac{1}{nK^2}}$&   all   $K\geq 1$\\
  \hline
\end{tabular}

\caption{A summary of comparisons between known results and our main results.
For both cases, $F$ is assumed to be $\mu$-strongly convex, and each $f_i$ is assumed to be convex and $L$-smooth.
For the sum of quadratics case each $f_i$ is additionally assumed to be quadratic. 
Here $K$ is the number of epochs (entire passes through the component indices $\{1,2,\dots ,n\}$), and $\kappa$ is the condition number of the problem. }
\label{table}
\end{table}

Recently, a few subsequent efforts have been made to characterize non-asymptotic convergence rates in terms of both $n$ and $K$.
Haochen and Sra~\cite{haochen2018random} develop the first \emph{non-asymptotic} convergence of $\OO{\nicefrac{\log^3(nK)}{(nK)^2} +\nicefrac{\log^4(nK)}{K^3}}$ under the condition $K \geq \Omega(\kappa \log (nK))$  for the same setting as \cite{gurbuzbalaban2015random} where $\kappa:=\nicefrac{L}{\mu}$ denotes the condition number.
However, compared with the baseline rate of SGD, i.e., $\OO{\nicefrac{1}{nK}}$, this convergence rate becomes an improvement only after $\omega(\sqrt{n})$ epochs.
This result is strengthened in  follow-up works that  demonstrate the superiority of without-replacement SGD:
\begin{enumerate}
    \item The work by Nagaraj, Jain, and Netrapalli~\cite{nagaraj2019sgd} considers a  more general setting where $f_i$'s no longer have to be quadratic, but just convex and smooth.
    Under this setting, they introduce clever coupling arguments to prove the non-asymptotic convergence rate of $\OO{\nicefrac{\log^2(nK)}{nK^2}}$ under the condition of $K \geq \Omega(\kappa^2 \log (nK))$ which is more stringent than that of \cite{haochen2018random}.
    Note that this rate is better than the baseline  of $\OO{\nicefrac{1}{nK}}$ as soon as the technical condition $K\geq \Omega(\kappa^2 \log (nK))$ is fulfilled.
  This upper bound is shown to be asymptotically tight up to poly-logarithmic factors by  Rajput, Gupta, and Papailiopoulos~\cite{rajput2020closing}. See Table~\ref{table} for comparisons.
  
    \item The convergence rate with a better dependency on $n$ for the strongly convex setting given by~\cite{nagaraj2019sgd} (at the cost of a severer requirement on $K$) has motivated researchers to revisit the quadratic sum case of \cite{gurbuzbalaban2015random} and obtain a convergence rate that has better dependency on $n$ than that of \cite{haochen2018random}.
    The first set of results in this direction are given by Safran and Shamir~\cite{safran2019good}, where an asymptotic lower bound of $\Om{\nicefrac{1}{(nK)^2} +\nicefrac{1}{nK^3}}$ is developed.
    They also establish a matching  upper bound for the $1$-dimensional case up to poly-logarithmic factors, evidencing that their lower bound is likely to have the correct dependency on $n$.
    The question of correct dependency  is settled  by  Rajput, Gupta, and Papailiopoulos~\cite{rajput2020closing} where the non-asymptotic convergence rate of $\OO{\nicefrac{\log^2(nK)}{(nK)^2} +\nicefrac{\log^3(nK)}{nK^3}}$ under the condition $K\geq \Omega(\kappa^2 \log (nK))$ is established  building on the coupling arguments in \cite{nagaraj2019sgd}. See Table~\ref{table} for comparisons.
\end{enumerate}

Despite such noticeable development over the years, there are key limitations shared among the existing non-asymptotic results (see Section~\ref{sec:limit} for precise details):
\begin{itemize}
    \item All the results place requirements\footnote{One notable exception is \cite[Theorem 2]{nagaraj2019sgd}. However, the result only proves the rate of $\OO{\nicefrac{\log(nK)}{nK}}$ which is not an improvement over the baseline rate of $\OO{\nicefrac{1}{nK}}$ for with-replacement SGD. } on the  number of epochs of the form $K \geq \kappa^{\alpha} \log (nK)$ for some constant $\alpha\geq 1$.
    \item All the non-asymptotic results have extra poly-logarithmic terms, raising a question whether one can remove them with improved analysis. 
\end{itemize}
In this paper, we aim to overcome these limitations.

\subsection{Summary of our main results} \label{sec:sum}
We establish tight convergence rates for without-replacement SGD (see Table~\ref{table}).
The key to obtaining tight convergence rates is to consider iteration-dependent step sizes $\OO{\nicefrac{1}{i}}$ for $i=1,\dots, n$, during the first epoch and the constant step size $\OO{\nicefrac{1}{nk}}$ for the $k$-th epoch ($k\geq 2$).
Our analysis builds on the per-iteration/epoch progress bounds developed in the prior arts~\cite{nagaraj2019sgd,rajput2020closing} (see Section~\ref{sec:per}).
The main distinction lies in turning those progress bounds into global convergence rate bounds.
One such tool for obtaining a non-asymptotic convergence rate is a  version of Chung's lemma~\cite[Lemma 1]{chung1954stochastic}, developed in the stochastic approximation literature.
Unfortunately, the original proof has some errors as pointed out by Fabian~\cite[Discussion above Lemma 4.2]{fabian1967stochastic}, and even assuming its correctness, it turns out that the lemma is not sufficient for obtaining the desired convergence rate bound (see Section~\ref{sec:fail}).
To overcome this difficulty, in Section~\ref{sec:tight}, we introduce a variant of Chung's lemma  (Lemma~\ref{main:2}) that can handle the case where there are two asymptotic parameters $n$ and $K$; this lemma may be  of independent interest. 
Our approach removes the epoch requirement of type $K\geq \kappa^{\alpha} \log (nK)$ as well as the extra poly-logarithmic terms in the convergence rates.

\subsection{Other related work}
Apart from the works mentioned above, there are few other theoretical studies of without-replacement SGD, although with different objectives.
Shamir~\cite{shamir2016without} demonstrates that without-replacement SGD is not  worse than SGD. 
His proof techniques use tools from  transductive learning theory, and as a result,  his results  only cover the first-epoch and the case where $F$ is a generalized linear function. 
As for another study, a recent work by Nguyen et al.~\cite{nguyen2020unified} studies some non-convex settings.
For the strongly convex case, they also obtain a convergence guarantee of $\OO{\nicefrac{1}{K^2}}$ under some weaker assumptions (for instance, $f_i$'s are not required to be convex).
However, this rate does not beat the baseline rate of SGD, i.e., $\OO{\nicefrac{1}{nK}}$.
 
\subsection{Problem setup and notation}
Given the cost function \eqref{def:F} of the finite-sum form, we consider the following  optimization problem:
\begin{align} \label{opt}
    \min_{x\in \re^d} F(x) = \frac{1}{n}\sum_{i=1}^n f_i(x)\,.
\end{align}
We call $F$ the cost function and $f_i$'s the component functions.
Also, let $\xs$ be the optimum solution of \eqref{opt}.

In solving the above optimization problem \eqref{opt}, we consider the without-replacement version of SGD  with the initial iterate $x_0\in \re^n$.
 For $k\geq 1$,  the $k$-th epoch (pass) is executed by first randomly shuffling the component functions with the permutation $\sigma_k$ on $\{1,2,\dots,n\}$ and going through each of them as 
 \begin{align}
     x_{k,i} = x_{k,i-1}- \ee_{k,i} \g_{\sigma_k(i)}(x_{k,i-1})\quad \text{ for $i=1,2,\dots,n$,}\label{update}
 \end{align} 
 where $x_{k,0} := x_{k-1,n}$ for all $k>1$ and   $x_{1,0}:=x_0$, and $\ee_k$ is the step size for the $i$-th iteration of the $k$-th epoch.
For simplicity, we denote the output of the $k$-th epoch by $y_k$, i.e., $y_k:= x_{k,n}= x_{k+1,0}$.

Lastly, we provide formal definitions of the regularity assumptions for functions.
For definitions, let $h:\re^d\to \re $ be a differentiable function and $G,L,\mm>0$ be some positive numbers.
\begin{definition}We say $h$ is $G$-Lipschitz if $\norm{\nabla h (x)}\leq G$ for all $x\in \re^d$. 
\end{definition}
\begin{definition}We say $h$ is $L$-smooth if $\norm{\nabla h(x)-\nabla h(y)}\leq L\norm{x-y}$ for any $x,y\in \re^d$.
\end{definition}
\begin{definition}We say $h$ is $\mu$-strongly convex if $h(y)\geq h(x) +\inp{\nabla h(x) }{y-x} + \frac{\mu}{2}\norm{y-x}^2$ for any $x,y\in \re^d$.
\end{definition}

Starting from the next section, we first focus on the case where $F$ is $\mu$-strongly convex and $f_i$'s are convex, $G$-Lipschitz and $L$-smooth.
Later in Section~\ref{sec:quad}, we will demonstrate how our techniques can be extended to obtain tight convergence rates for the case where $f_i$'s are additionally assumed to be quadratic.

\section{Preliminaries: existing per-iteration/-epoch bounds}
\label{sec:per}
We first need to quantify the progress made by the algorithm over each iteration.
For  without-replacement SGD, there are two different types of analyses: 
\begin{enumerate}
    \item Per-iteration analysis where one characterizes the progress made at each  iteration.
    \item Per-epoch analysis where one characterizes the aggregate progress made over one epoch.
\end{enumerate}

For per-iteration analysis, Nagaraj, Jain, and Netrapalli~\cite{nagaraj2019sgd} develop  coupling arguments to prove that the progress made by without-replacement SGD is \emph{not} worse than with-replacement SGD.
In particular, their coupling arguments demonstrate the closeness in expectation between the iterates of  without- and with-replacement SGD.
The following is a consequence of their coupling argument:
\begin{proposition}[Per-iteration analysis {\cite[implicit in Section A.1]{nagaraj2019sgd}}] \label{per:0}
Assume  for $L,G,\mm>0$ that each component function $f_i$ is convex, $G$-Lipschitz and $L$-smooth and the cost function $F$ is $\mm$-strongly convex.
Then, for any step size for the $(i+1)$-th iteration of the $k$-th epoch such that $\ee_{k,i+1}\leq \frac{2}{L}$, the following bound holds between the adjacent iterates:
\begin{align}
    \ex{}{\norm{x_{k,i+1}-\xs}^2}&\leq \left( 1- \ee_{k,i+1}\mm/2 \right)\cdot \ex{}{\norm{x_{k,i}-\xs}^2}+ 3\ee_{k,i+1}^2 G^2  + 4\ee_{k,i+1}^3 \kappa L
    G^2\,.  \label{perbd:0}
 \end{align}
where the expectation is taken over the randomness within the $k$-th epoch.
\end{proposition}

However, with the above analysis, one can only obtain results comparable to with-replacement SGD, as manifested in \cite[Theorem 2]{nagaraj2019sgd}.
In order to characterize better progress, one needs to characterize the aggregate progress made over one epoch as a whole.
A nice property of without-replacement SGD when considered over one epoch is the following  observation  due to Nedi{\'c} and Bertsekas~\cite[Chapter 2]{nedic2001convergence}.
For each epoch, say the $k$-th epoch, assuming that the iterates $\{x_{k,i}\}_{i=1}^n$ stay close to the initial iterate $x_{k,0}$, the aggregate update direction will closely approximate the \emph{full} gradient at $x_{k,0}$, i.e.,
\begin{align} \label{obs:per}
    \sum_{i=1}^n\g_{\sigma_k(i)} (x_{k,i})\approx \sum_{i=1}^n\g_{\sigma_k(i)} (x_{k,0})= \sum_{i=1}^n\g_{i} (x_{k,0}) = nF(x_{k,0})\,.
\end{align}
Based on this observation, together with the coupling arguments,  Nagaraj, Jain, and Netrapalli!\cite{nagaraj2019sgd} obtain the following improved bound for one epoch as a whole:
\begin{proposition}[Per-iteration analysis {\cite[implicit Section 5.1]{nagaraj2019sgd}}] \label{per:1}
Under the same setting as Proposition~\ref{per:0}, let $\ee_{k+1}\leq \frac{2}{L}$ be the step size for the $(k+1)$-th epoch, i.e., $\ee_{k+1,i}\equiv \ee_{k+1}$ for $i=1,2,\dots, n$.
Then, the following bound holds between the output of the $(k+1)$-th and $k$-th epochs  $y_{k+1}$ and $y_k$: 
\begin{align}
\begin{split}
    \ex{}{\norm{y_{k+1}-\xs}^2}&\leq \left( 1- 3n\ee_{k+1} \mm /4  + n^2 \ee_{k+1}^2L^2   \right)\cdot \norm{y_k-\xs}^2 \\
    &\quad-2n\ee_{k+1}\left( 1-4n\ee_{k+1}\kappa L\right) \cdot (\ex{}{F(y_k)-F(\xs)})+ 20n^2\ee_{k+1}^3\kappa L G^2   + 5n^3\ee_{k+1}^4 L^2G^2\,.
    \end{split} \label{perbd:1}
\end{align}
where the expectation is taken over the randomness within the $(k+1)$-th epoch.
\end{proposition}

Having these per-iteration/-epoch progress bounds, the final ingredient of the non-asymptotic convergence rate analysis is to turn these bounds into \emph{across-epochs} global convergence bounds. 

\section{Limitations of the previous approach}
\label{sec:limit}
In this section, we explain the approach used in the previous works~\cite{haochen2018random,nagaraj2019sgd, rajput2020closing} to obtain non-asymptotic bounds and illustrate the limitations of their approach.
To turn the per-iteration/-epoch progress bounds \eqref{perbd:0} and \eqref{perbd:1} into non-asymptotic convergence rates, the previous works take a simple approach of adopting  constant stepsizes. 
For instance, we illustrate the approach in \cite{nagaraj2019sgd} as follows:

\noindent  {\bf Illustration of previous approach:} In Proposition~\ref{per:1}, let us  choose the constant step size $\ee_{k,i}\equiv \ee$.
Since $\ex F(y_k)-F(\xs)>0$, one can disregard the second term in the upper bound \eqref{per:1} as long as $4n\ee\kappa L < 1$.
Suppose that we choose $\ee$ small enough that $4n\ee\kappa L < 1$ holds.
Then, since we also have $\frac{n\ee  \mm}{4}  > n^2 \ee^2L^2$, the per-epoch bound \eqref{perbd:1}  becomes:
\begin{align}
   \ex{}{\norm{y_{k+1}-\xs}^2}&\leq  \left( 1-n\ee  \mm /2    \right)\cdot \ex{}{\norm{y_k-\xs}^2}  + 20n^2\ee^3\kappa L G^2   + 5n^3\ee^4 L^2G^2\,. \label{bd:ex2}
\end{align} 
Now, recursively applying \eqref{bd:ex2} for $k=0,1,\dots,K-1$, we obtain the following bound:
\begin{align*}
    \ex{}{\norm{y_K-\xs}^2}  &\leq \left( 1- n\ee \mm /2  \right)^{K}\cdot \ex{}{\norm{x_{0}-\xs}^2}+ \sum_{t=0}^\infty \left( 1- n\ee \mm/2   \right)^t \left[ 20n^2\ee^3\kappa L G^2   + 5n^3\ee^4 L^2G^2\right]\\
    &\leq \exp\left( - n\ee \mm K/2 \right)\cdot  \norm{x_{0}-\xs}^2+   40n\ee^2\kappa^2 G^2   + 10n^2\ee^3 \kappa LG^2\,.
\end{align*}
 Having established this, they choose  $\ee =   \alpha \cdot \frac{2\log (nK)}{\mm nK} $ for some $\alpha>3$ to obtain:
\begin{align*}
    \ex{}{\norm{y_K-\xs}^2}\leq \frac{\norm{x_{0}-\xs}^2}{(nK)^\alpha}+      \frac{80\alpha^2 \kappa^2 G^2 \log^2 (nK)}{\mu^2 nK^2}  +   \frac{40\alpha^3 \kappa L G^2\log^3(nK)}{\mm^3 nK^3}\,.
\end{align*}
One can easily see that the assumption $4n\ee\kappa L < 1$ is satisfied whenever $K \geq 8\alpha \kappa^2  \log(nK)$.\qed

The limitations of the constant step size approach are manifested in the above illustration: (i) it incurs extra  poly-logarithmic terms in the convergence rate bound and (ii) requires the number of epochs to be sufficiently large.
Indeed,  all previous works share these limitations.
Having noticed the limitations, one might then wonder if one can obtain better results by abandoning constant step size.  
\section{Chung's lemma: analytic tools for varying stepsize} \label{sec:warmup}
As an effort to overcome the limitations, let us allow step sizes to vary across iterations or epochs. 
Let us first consider the per-iteration progress bound in Proposition~\ref{per:0}.
Since Proposition~\ref{per:0} works for any iterations, one can disregard the epoch structure and simply  denote by $x_t$ the $t$-th iterate and by $\ee_{t}$ the step size used for the $t$-th iteration.
Choosing  $\ee_{t}= \frac{2\alpha}{\mm} \cdot \frac{1}{\ini+t}$ for all $t\geq 1$ with the initial index $\ini$, where we choose $\ini=\alpha \cdot \kappa$ to ensure $\ee_{t}\leq \frac{2}{L}$, the per-iteration bound \eqref{perbd:0} becomes (we also use $\ee_{t+1}^3\leq \ee_{t+1}^2 \frac{L}{2}$):
\begin{align}
    \ex{}{\norm{x_{t+1}-\xs}^2}&\leq \left( 1- \frac{\alpha}{\ini+t+1}  \right)\cdot \ex{}{\norm{x_{k,i}-\xs}^2}+ \frac{\alpha^2 G^2 (12\mm^{-2} +32\kappa^3)}{(\ini+t+1)^2} \,.  \label{bd:vary}
\end{align}
In fact, for the bounds of type \eqref{bd:vary}, there are suitable tools for obtaining convergence rates: versions of \emph{Chung's lemma}~\cite{chung1954stochastic}, developed in the stochastic approximation literature. 
Among the various versions of Chung's lemma, there is one non-asymptotic version~\cite[Lemma 1]{chung1954stochastic}:
\begin{lembox}[Non-asymptotic Chung's lemma] \label{main:1}
Let $\{\ab_{k}\}_{k\geq 0}$ be a sequence of positive real numbers.
Suppose that there exist an initial index $\ini >0$ and real numbers $\ccc>0$, $\alpha >\beta>0$ and $\eps>0$ such that $\ab_{k+1}$  satisfies  the following inequality:
\begin{align} \label{adj}
    \ab_{k+1} \leq  \exp\left( -\frac{\alpha }{\ini+k+1}  \right) \ab_k +\frac{\ccc}{(\ini+k+1)^{\beta +1}}\quad \text{  for any integer $k\geq 0$}\,. 
\end{align}
Then, for any $K\geq 1$ we have the following bound:
\begin{align} \label{bound:1}
   \ab_K &\leq \exp\left( -\alpha\cdot \sum_{i=1}^K \frac{1}{\ini+i}  \right)\cdot \ab_0 +   \frac{\frac{1}{\alpha-\beta}   e^{\frac{\alpha}{\ini+1}}\cdot \ccc}{(\ini+K)^\beta} +  \frac{  e^{\frac{\alpha}{\ini+ 1}}\cdot \ccc}{(\ini+K)^{\beta+1}}\\
   &\leq \frac{(\ini+1)^\alpha}{(\ini+K)^\alpha}\cdot \ab_0 +   \frac{\frac{1}{\alpha-\beta}   e^{\frac{\alpha}{\ini+1}}\cdot \ccc}{(\ini+K)^\beta} +  \frac{  e^{\frac{\alpha}{\ini+ 1}}\cdot \ccc}{(\ini+K)^{\beta+1}}\,. \label{bound:12}
\end{align}
\end{lembox}
\begin{proof}
Unfortunately, the original ``proof'' contains some errors as pointed out by Fabian~\cite[Discussion above Lemma 4.2]{fabian1967stochastic}.
We are able to correct the original proof; for this, see Section~\ref{sec:correct}.
\end{proof} 
Let us apply Lemma~\ref{main:1} to \eqref{bd:vary} as a warm-up.
From \eqref{bd:vary}, one can see that $\ccc$ in Lemma~\ref{main:1} can be chosen as $G^2 (12\mm^{-2} +32\kappa^3)$. Hence, we obtain:
\begin{corollary} \label{cor:ex}
Under the setting of Proposition~\ref{per:0}, let $\alpha>1$ be a constant, and consider the step size $\ee_{k,i}=\frac{2\alpha/\mm}{\ini+n(k-1)+i}$ for  $\ini:= \alpha \cdot \kappa$. Then the following convergence rate holds for any $K\geq 1$:
\begin{align}
    \ex\norm{y_K-\xs}^2 \leq  \frac{(\ini+1)^{\alpha }\norm{x_0-\xs}^2}{(\ini+nK)^{\alpha }} + \frac{\frac{e}{{\alpha}-1}\alpha^2 G^2 (12\mm^{-2} +32\kappa^3)}{\ini+nK} + \frac{e\alpha^2 G^2 (12\mm^{-2} +32\kappa^3)}{(\ini+nK)^2}\,. \label{bd:cor}
\end{align}
\end{corollary}
\noindent Notably, Corollary~\ref{cor:ex} is an improvement over \cite[Theorem 2]{nagaraj2019sgd} as it gets rid of   extra poly-logarithmic terms.
Having this successful example, one might  wonder if  this lemma can be used to improve the constant step approach from Section~\ref{sec:limit}. 
\section{An illustrative failed attempt using Chung's lemma} \label{sec:fail}

Let us apply Lemma~\ref{main:1} to Proposition~\ref{per:1}.
For illustrative purpose, consider an ideal situation where instead of the actual progress bound \eqref{perbd:1}, a nice epoch progress bound of the form \eqref{bd:ex2} holds:
\begin{align}
   \ex{}{\norm{y_{k+1}-\xs}^2}&\leq  \left( 1-n\ee_{k+1}  \mm /2    \right)\cdot \ex{}{\norm{y_k-\xs}^2}  + 20n^2\ee_{k+1}^3\kappa L G^2   + 5n^3\ee_{t+1}^4 L^2G^2\,. \label{bd:ideal}
\end{align} 
Following the same principle as the previous section, let us take $\ee_k = \frac{2\alpha/\mm}{\ini+ nk}$  for some constant $\alpha>2$.
On the other hand, to make things simpler, let us assume that one can take $\ini=0$. 
Plugging this stepsize into \eqref{bd:ideal},  we obtain the following bound for some constants $c >0$:
\begin{align*}
   \ex{}{\norm{y_{k+1}-\xs}^2}&\leq  \left( 1-\frac{\alpha }{ k+1}    \right)\cdot \ex{}{\norm{y_k-\xs}^2}  + \frac{c/n }{(k+1)^3}   \,,  
\end{align*} 
which then yields the following non-asymptotic bound due to  Lemma~\ref{main:1}:
\begin{align}
    \ex{}{\norm{y_{K}-\xs}^2}&\leq \OO{\frac{1}{K^{\alpha}}}+\OO{\frac{1}{nK^2}} + \OO{\frac{1}{nK^3}} \,. \label{bd:hyp}
\end{align}
Although the last two terms in \eqref{bd:hyp} are what we desire (see Table~\ref{table}), the  first term is undesirable.
Even though we choose $\alpha$ large, this bound will still contain the term $O( \nicefrac{1}{K^{\alpha}} )$ which is an obstacle when one tries to show the superiority over the baseline of $\OO{\frac{1}{nK}}$; note that the former is a better rate than the baseline only if $K\geq \Omega(n^{\frac{1}{\alpha-1}})$.
Therefore, for the target convergence bound, one needs other versions of Lemma~\ref{main:1}.

\section{A variant of Chung's lemma and tight convergence rates} \label{sec:tight}

As we have seen in the previous section, Chung's lemma is not enough for capturing the desired convergence rate.
In this section, to capture the right order for both  $n$ and $K$,  we develop a variant of Chung's lemma.

\begin{lembox} \label{main:2}
Let $n>0$ be an integer, and  $\{\ab_{k}\}_{k\geq 0}$ be a sequence of positive real numbers.
Suppose that there exist an initial index $\ini >0$ and real numbers  $\co, \cbo >0$, $\alpha >\beta>0$ and $\eps>0$ such that  the following are satisfied:
\begin{align} 
     \ab_{1} &\leq  \exp\left(-\alpha \sum_{i=1}^{n }\frac{1}{\ini+i}\right)  \ab_{0}  +\co  \quad \text{and} \label{11}\\
        \ab_{k+1} &\leq  \exp \left( -\alpha  \sum_{i=1}^{n } \frac{1}{\ini+nk+i} + \frac{\eps}{k^2}\right)  \ab_k +   \frac{\cbo}{(\ini+n(k+1))^{\beta+1}}\quad \text{  for any integer $k\geq 1$}\,.  \label{kk}
\end{align}
Then, for any $K\geq 1$ we have the following bound for $c:=e^{\eps\pi^2/6 }$:
\begin{align} \label{bound:2}
   \ab_K \leq  \frac{c(\ini+1)^\alpha}{(\ini+nK)^\alpha}\cdot \ab_0 +   \frac{ c\cdot (\ini+n+1)^\alpha\cdot \co}{(\ini+nK)^\alpha}+  \frac{\frac{c}{\alpha-\beta}   e^{\frac{\alpha}{\ini+n+1}}\cdot \cbo}{n(\ini+nK)^\beta} +  \frac{c  e^{\frac{\alpha}{\ini+n+1}}\cdot \cbo}{(\ini+nK)^{\beta+1}}\,.
\end{align}
\end{lembox} 
\begin{proof}[Proof of Lemma~\ref{main:2}]
 See Section~\ref{pf:var}.
\end{proof}
\subsection{Tight convergence rate for strongly convex costs}
Now we use Lemma \ref{main:2} to obtain a tight convergence rate.
Let $\ab_k:=   \ex{}{\norm{y_k-\xs}^2}$ for $k\geq 1$ and $\ab_0:=\norm{x_0-\xs}^2$.
Let $\alpha>2$ be an arbitrarily chosen constant.
For the first epoch, we take the following iteration-varying step size: $\ee_{1,i}= \frac{2\alpha}{\mm} \cdot \frac{1}{\ini+i}$, where $\ini=\alpha \cdot \kappa$ to ensure $\ee_{1,i}\leq \frac{2}{L}$.
Then, similarly to Corollary~\ref{cor:ex}, yet this time by using the bound \eqref{bound:1} in Lemma~\ref{main:1}, one can derive the  the following bound:
 \begin{align}
     \ab_1&\leq \exp\left( -\alpha \cdot \sum_{i=1}^n \frac{1}{\ini+i}  \right)\cdot \ab_0 +   \frac{a_1}{\ini+n}  \,, \label{bd:11}
 \end{align}
 where $a_1:=\alpha^2 G^2\cdot [\frac{e}{{\alpha}-1} (12\mm^{-2} +32\kappa^3) + e\alpha^2 G^2 (12\mm^{-1}L^{-1} +32\kappa^2)]$, i.e., $a_1=\OO{\kappa^3}$.

Next, let us establish bounds of the form  \eqref{kk} for the $k$-th epoch for $k\geq 2$.
From the second epoch on, we use the same step size within an epoch.
More specifically, for the $k$-th epoch we choose $\ee_{k,i}\equiv \ee_k =\frac{2\alpha/\mm}{\ini + nk}$.
Using similar argument to obtain \eqref{bd:ex2} in Section~\ref{sec:limit}, Proposition~\ref{per:1} yields the following bound for $k\geq 8\alpha \kappa^2-1$:
\begin{align}
  \ab_{k+1}&\leq  \exp\left( -n\ee_{k+1}  \mm /2    \right)\cdot \ab_k  + 20n^2\ee_{k+1}^3\kappa L G^2   + 5n^3\ee_{k+1}^4 L^2G^2\,. \label{bd:01}
\end{align} 
For $k< 8\alpha \kappa^2-1$, recursively applying Proposition~\ref{per:0}  with the fact $(n\ee_{k+1})^{-1}\leq 4\kappa L +L/(2n)$ implies:
\begin{align}
    \ab_{k+1}&\leq \exp\left( - n\ee_{k+1}\mm/2 \right)\cdot \ab_k+ 3n^2\ee_{k+1}^3G^2 (4\kappa L +L/(2n))  + 4n\ee_{k+1}^3 \kappa L
    G^2\,.  \label{bd:02}
 \end{align}
 Therefore, combining \eqref{bd:01} and \eqref{bd:02}, we obtain the following bound which holds for any $k\geq 1$:
 \begin{align}
    \ab_{k+1}&\leq \exp\left( - n\ee_{k+1}\mm/2 \right)\cdot\ab_k + a_2\cdot n^2\ee_{k+1}^3 \,,  \label{bd:03}
 \end{align}
 where $a_2:= 12\kappa L G^2 + (3L/2+4\kappa LG^2)/n+20\kappa LG^2 + 5\mu^2G^2/8$, i.e., $a_2=\OO{\kappa}$.
 Let us modify the coefficient of $\ab_k$ in \eqref{bd:03} so that it fits into the form of \eqref{kk} in Lemma~\ref{main:2}.
 First $ \exp\left( - n\ee_{k+1}\mm/2 \right) =\exp\left( - \alpha n/(\ini+n(k+1)) \right)$.
 This expression can be modified as 
 \begin{align*}
     \exp\left[-\alpha  \cdot\sum_{i=1}^n  \frac{1}{\ini+nk+i}+ \alpha \cdot\sum_{i=1}^n \left(  \frac{1}{\ini+nk+i}- \frac{1}{\ini+n(k+1)} \right) \right]\,,
\end{align*}     
which is then upper bounded by $\exp\left[-\alpha  \cdot\sum_{i=1}^n  \frac{1}{\ini+nk+i}+ \frac{\alpha }{k^2} \right]$.
Thus, \eqref{bd:03} can be rewritten as:
 \begin{align}
     \ab_{k+1} &\leq  \exp\left(-\alpha \cdot\sum_{i=1}^n  \frac{1}{\ini+nk+i}+ \frac{\alpha }{k^2}\right)\cdot \ab_k + \frac{8a_2 \alpha^3 n^2\mu^{-3}}{(\ini + n(k+1))^3}\,. \label{bd:kk} 
 \end{align}
 Now applying Lemma~\ref{main:2} with \eqref{bd:11} and \eqref{bd:kk} implies the following result:
\begin{thmbox}[Strongly convex costs]
\label{thm:1}
Assume  for $L,G,\mm>0$ that each component function $f_i$ is convex, $G$-Lipschitz and $L$-smooth and the cost function $F$ is $\mm$-strongly convex.
For  any constant $\alpha>2$, let $\ini:= \alpha \cdot \kappa$, and consider the step sizes $\ee_{1,i} =  \frac{2\alpha/\mm}{\ini+i}$ for $i=1,\dots, n$ and  for $k>1$, $\ee_{k,i}=  \frac{2\alpha/\mm}{\ini+nk}$ for $i=1,2,\dots, n$.
Then, the following convergence bound holds for any $K\geq 1$: 
\begin{align} \label{bd:master}
     \ex{}{\norm{y_{K}-\xs}^2} \leq  \frac{c_1 \cdot n}{(\ini+nK)^2} + 
     \frac{c_2 \cdot (\ini+n+1)^{\alpha-1}}{(\ini+nK)^{\alpha}} +\frac{c_3\cdot  \norm{y_{0}-\xs}^2}{(\ini+nK)^{\alpha}}   \,,
 \end{align}
 where $c_1=\OO{\kappa^4}$, $c_2=\OO{\kappa^3}$, and $c_3=\OO{\kappa^\alpha}$.
\end{thmbox}

\subsection{Tight convergence rate for quadratic costs}
 \label{sec:quad}
 
 Similarly, yet using some improved per-epoch progress bound developed in \cite{rajput2020closing}, one can prove the following better convergence rate for the case where $F$ is additionally assumed to be quadratic.
Note that this case is slightly more general than the  setting  considered in \cite{gurbuzbalaban2015random,haochen2018random} where each $f_i$ is assumed to be quadratic.
The details are deferred to Appendix~\ref{appen:quad}.
\begin{thmbox}[Quadratic costs]
\label{thm:2}
Under the setting of Theorem~\ref{thm:1}, we additionally assume that $F$ is quadratic.
For  any constant $\alpha>4$, let $\ini:= \alpha \cdot \kappa$, and consider the step sizes $\ee_{1,i} =  \frac{2\alpha/\mm}{\ini+i}$ for $i=1,\dots, n$ and  for $k>1$, $\ee_{k,i}=  \frac{2\alpha/\mm}{\ini+nk}$ for $i=1,2,\dots, n$.
Then, the following convergence bound holds for any $K\geq 1$: 
\begin{align} \label{bd:master2}
     \ex{}{\norm{y_{K}-\xs}^2} \leq  \frac{c_1 \cdot n^2}{(\ini+nK)^3}  +  \frac{c_2}{(\ini+nK)^2}+  
     \frac{c_3 \cdot (\ini+n+1)^{\alpha-1}}{(\ini+nK)^{\alpha}} +\frac{c_4\cdot  \norm{y_{0}-\xs}^2}{(\ini+nK)^{\alpha}}   \,,
 \end{align}
 where $c_1 =\OO{\kappa^6}$, $c_2=\OO{\kappa^4}$, $c_3=\OO{\kappa^3}$, and $c_4=\OO{\kappa^\alpha}$.
\end{thmbox}

\section{Proofs of the versions of Chung's lemma (Lemmas~\ref{main:1} and \ref{main:2})} \label{sec:correct}
  We begin with an elementary fact that we will use throughout:
\begin{proposition}[Integral approximation; see e.g. {\cite[Theorem~14.3]{lehman2010mathematics}})] \label{pro:approx}
Let $f:\re^+ \to \re^+$ be a non-decreasing continuous function.
Then, for any integers $1\leq m <n$, $\int_m^n f(x)dx +f(m) \leq \sum_{i=m}^n f(i) \leq \int_m^n f(x)dx +f(n)$.
Similarly, if $f$ is non-increasing, then for any integers $1\leq m <n$, $\int_m^n f(x)dx +f(n) \leq \sum_{i=1}^n f(i) \leq \int_1^n f(x)dx +f(m)$.
\end{proposition}  

\noindent We first prove Lemma~\ref{main:1}, thereby coming up with a correct proof of the non-asymptotic Chung's lemma.
\subsection{A correct proof of Chung's lemma (Proof of Lemma~\ref{main:1})}
 For simplicity, let us define the following quantities for $k\geq 1$: 
\begin{align*}
a_{k}:=\exp\left(- \frac{\alpha }{k_0+ k}  \right) ~~\text{and}~~ c_k:=\frac{\ccc}{(\ini+  k)^{\beta+1}}\,.
\end{align*}  
Using these notations,  the recursive relation \eqref{adj}  becomes:
\begin{align} 
\ab_{k+1} &\leq  a_{k+1} \cdot \ab_k +c_{k+1} \quad \text{for any integer $k\geq 1$.} \label{adj2}
\end{align}
After recursively applying \eqref{adj2} for $k=0,1,2,\dots, K-1$, one obtain the following bound:
\begin{align}\label{6:int2}
    \ab_K \leq \ab_0  \prod_{j=1}^{K}a_j   + \left(\prod_{j=1}^{K}a_j \right)\cdot\left[\sum_{k=1}^K  \left(\prod_{j=1}^{k}a_j \right)^{-1} c_{k} \right] \,.
\end{align}
Now let us upper and lower bound the product of $a_j$'s.
Note that 
\begin{align*}
    \prod_{j=1}^ka_j = \exp \left(-\alpha \sum_{i=1}^{k} \frac{1}{\ini+i}\right)\quad \text{for any $k\geq 1$,}
\end{align*} 
and hence  Proposition~\ref{pro:approx} with $f(x) = \frac{1}{\ini+x}$ implies 
\begin{align} \label{6:est1}
    e^{-\frac{\alpha}{\ini+1}}\left(\frac{\ini+1}{\ini+k} \right)^\alpha  \leq \prod_{j=1}^ka_j \leq \left(\frac{\ini+1}{\ini+k} \right)^\alpha \,.
\end{align}
Therefore, we have
\begin{align*}
   \sum_{k=1}^K  \left(\prod_{j=1}^{k}a_j\right)^{-1} c_{k} &\leq e^{\frac{\alpha}{\ini+1}} \sum_{k=1}^K \left(\frac{\ini+k}{\ini+1} \right)^\alpha  \cdot \frac{\ccc}{(\ini+k)^{\beta+1}} = \frac{ e^{\frac{\alpha}{\ini+1}} \cdot \ccc}{(\ini+1)^\alpha}\cdot \sum_{k=1}^K  (\ini+k)^{\alpha-\beta-1} \,.
\end{align*}
Applying Proposition~\ref{pro:approx} with $f(x)= (\ini+x)^{\alpha-\beta-1}$ to the above, since $\frac{1}{\alpha-\beta}(\ini+x)^{\alpha-\beta}$ is an anti-derivative of $f$, we obtain the following upper bounds:
\begin{align*}\frac{ e^{\frac{\alpha}{\ini+1}} \cdot \ccc}{(\ini+1)^\alpha}\cdot  \begin{cases}
   \frac{1}{\alpha-\beta}\left( (\ini+K)^{\alpha-\beta} -  (\ini+1)^{\alpha-\beta}\right) + (\ini+K)^{\alpha-\beta-1} , &\text{if }\alpha>\beta+1,\\  
   K, &\text{if }\alpha=\beta+1,\\ 
   \frac{1}{\alpha-\beta}\left( (\ini+K)^{\alpha-\beta} -  (\ini+1)^{\alpha-\beta}\right) + (\ini+1)^{\alpha-\beta-1}  &\text{if }\alpha<\beta+1.   \end{cases}  
\end{align*}
Combining all three cases, we conclude:
\begin{align*}
     \sum_{k=2}^K  \left(\prod_{j=2}^{k}a_j\right)^{-1} c_{k} &\leq \frac{ e^{\frac{\alpha}{\ini+1}} \cdot \ccc}{(\ini+1)^\alpha}\cdot \left( \frac{(\ini+K)^{\alpha-\beta}}{\alpha-\beta}+ (\ini+K)^{\alpha-\beta-1}\right)\,.
\end{align*}
Plugging this back to \eqref{6:int2} and using  \eqref{6:est1} to upper bound $\prod_{j=1}^{K}a_j$, we obtain \eqref{bound:1}.
Using \eqref{6:est1} once again to upper bound the coefficient of $\ab_0$,  we obtain \eqref{bound:12}. Hence, the proof is completed.
 
\subsection{Proof of Lemma~\ref{main:2}} \label{pf:var}
 
 The proof is generally analogous to that of Lemma~\ref{main:1}, while some distinctions are required to capture the correct asymptotic for the two parameters $n$ and $K$.
 To simplify notations, let us define the following quantities for $k\geq 1$: 
\begin{align*}
a_{k}:=\exp\left(- \alpha \cdot \sum_{i=1}^n \frac{1}{\ini+n(k-1)+i}  \right),~~ b_k:=\exp\left(\frac{\eps}{(k-1)^2} \right),~~\text{and}~~ c_k:=\frac{\cbo}{(\ini+nk)^{\beta+1}}\,.
\end{align*}  
Using these notations,  the recursive relations \eqref{11} and \eqref{kk} become:
\begin{align}
\ab_1 &\leq a_1\cdot \ab_0 +\co \label{112}\\
     \ab_{k+1} &\leq  a_{k+1}b_{k+1}\cdot \ab_k +c_{k+1} \quad \text{for any integer $k\geq 1$.} \label{kk2}
\end{align}
Recursively applying \eqref{kk2} for $k=1,2,\dots, K-1$ and then \eqref{112}, we obtain:
\begin{align}\label{int}
    \ab_K \leq a_1\ab_0  \prod_{j=2}^{K}a_jb_j  + \left(\prod_{j=2}^{K}a_jb_j\right)\cdot\left[ \co+ \sum_{k=2}^K  \left(\prod_{j=2}^{k}a_jb_j\right)^{-1} c_{k} \right] \,.
\end{align}
Again one can use the fact $\sum_{i\geq 1}{i^{-2}} =\frac{\pi^2}{6}$ to upper and lower bound the product of $b_j$'s:
\begin{align}
    1\leq \prod_{i=2}^Kb_i \leq \exp \left(\sum_{i=2}^K\frac{\eps}{(i-1)^2}\right) \leq \exp\left( \eps \pi^2/6\right)\,. \label{prod:b}
\end{align}
Applying \eqref{prod:b} to \eqref{int}, we obtain the following bound (recall $c= e^{ \eps \pi^2/6}$):
\begin{align}\label{int2}
    \ab_K \leq c \ab_0 \prod_{j=1}^{K}a_j  + c \prod_{j=2}^{K}a_j\cdot\left[ \co+ \sum_{k=2}^K  \left(\prod_{j=2}^{k}a_j\right)^{-1} c_{k} \right]  \,.
\end{align} 
To obtain upper and lower bounds on the product of $a_j$'s, again note for any $1\leq \ell\leq k$ that
\begin{align*}
    \prod_{j=\ell}^ka_j = \exp \left(-\alpha \cdot \sum_{i=1}^{(k-\ell+1)n} \frac{1}{\ini+n(\ell-1)+i}\right)\,,
\end{align*} which can then be estimated as follows using Proposition~\ref{pro:approx}:
\begin{align} \label{est1}
   e^{-\frac{\alpha}{\ini+n(\ell-1)+1}} \left(\frac{\ini+n(\ell-1)+1}{\ini+nk} \right)^\alpha  \leq \prod_{j=\ell}^ka_j \leq \left(\frac{\ini+n(\ell-1)+1}{\ini+nk} \right)^\alpha\,.
\end{align}
Therefore, we have
\begin{align*}
   \sum_{k=2}^K  \left(\prod_{j=2}^{k}a_j\right)^{-1} c_{k} &\leq   e^{\frac{\alpha}{\ini+n+1}}\sum_{k=2}^K \left(\frac{\ini+nk}{\ini+n+1} \right)^\alpha  \cdot \frac{\cbo}{(\ini+nk)^{\beta+1}} = \frac{  e^{\frac{\alpha}{\ini+n+1}}\cdot \cbo}{(\ini+n+1)^\alpha}\cdot \sum_{k=2}^K  (\ini+nk)^{\alpha-\beta-1} \,.
\end{align*}
Applying Proposition~\ref{pro:approx} with $f(x)= (\ini+nx)^{\alpha-\beta-1}$ to the above, since $\frac{1}{n(\alpha-\beta)}(\ini+nx)^{\alpha-\beta}$ is an anti-derivative of $f$, we obtain the following upper bounds:
\begin{align*}
    \frac{ e^{\frac{\alpha}{\ini+n+1}}\cdot \cbo}{(\ini+n+1)^\alpha}\cdot  \begin{cases}
   \frac{1}{n(\alpha-\beta)}\left( (\ini+nK)^{\alpha-\beta} -  (\ini+2n)^{\alpha-\beta}\right) + (\ini+nK)^{\alpha-\beta-1} , &\text{if }\alpha>\beta+1,\\  
   K-1, &\text{if }\alpha=\beta+1,\\ 
   \frac{1}{n(\alpha-\beta)}\left( (\ini+nK)^{\alpha-\beta} -  (\ini+2n)^{\alpha-\beta}\right) + (\ini+2n)^{\alpha-\beta-1}  &\text{if }\alpha<\beta+1.   \end{cases}  
\end{align*}
Combining all three cases, we conclude:
\begin{align*}
     \sum_{k=2}^K  \left(\prod_{j=2}^{k}a_j\right)^{-1} c_{k} &\leq \frac{e^{\frac{\alpha}{\ini+n+1}}\cdot \cbo}{(\ini+n+1)^\alpha}\cdot\left( \frac{(\ini+nK)^{\alpha-\beta}}{n(\alpha-\beta)}+ (\ini+nK)^{\alpha-\beta-1}\right)\,.
\end{align*}
Plugging this back to \eqref{int2}, and using \eqref{est1} to upper bound the product of $a_j$'s,  the proof is completed.

\section{Conclusion}

Motivated by some limitations in the previous efforts, this work establishes tight convergence rates for without-replacement SGD.
The key to obtaining tight rates is to depart from the constant step size in the previous works and adopt time-varying step sizes.
We first observe that known tools for obtaining convergence rates such as Chung's lemma  are not suitable for our case where there are two parameters of interests, namely
$n$ and $K$, in the convergence bound.
To overcome the issue, we develop a variant of Chung's lemma that establishes the convergence bound with correct dependency on the two parameters.

	\bibliographystyle{alpha}
	\bibliography{ref}

 \appendix

 \section{Details for  the quadratic cost} \label{appen:quad}

\noindent Now let us use again Lemma \ref{main:2} to obtain a tight convergence rate.
We follow notations in Section~\ref{sec:tight}.
Again, we use \eqref{bd:11} for the first recursive inequality \eqref{11} in Lemma~\ref{main:2}.

For the second recursive inequalities \eqref{kk} in Lemma~\ref{main:2}, in order to obtain better convergence rate,  we use the following improved per-epoch bound for quadratic costs due to  Rajput, Gupta, and Papailiopoulos~\cite{rajput2020closing}:
\begin{proposition}[{\cite[implicit in Appendix A]{rajput2020closing}}] \label{per:2}
Under the setting of Proposition~\ref{per:0}, assume further that  $F$ is  quadratic. 
Then for any step size for the $(k+1)$-th epoch $\ee_{k+1}\leq \frac{2}{L}$, the following bound holds between the output of the $(k+1)$-th and $k$-th epochs  $y_{k+1}$ and $y_k$:   
\begin{align}
\begin{split}
    \ex{}{\norm{y_{k+1}-\xs}^2}  &\leq \left( 1-  3n\ee_{k+1} \mm/2 +5n^2 \ee_{k+1}^2 L^2   + 8 n^3 \ee_{k
    +1}^3 \kappa L^3  \right) \norm{y_k-\xs}^2 \\
    &\quad+10 n^3 \ee_{k+1}^4 L^2 G^2+40n^4\ee_{k+1}^5\kappa L^3G^2+ 32n\ee_{k+1}^3 \kappa L G^2\,.
\end{split} \label{perbd:2}
\end{align}
where the expectation is taken over the randomness within the $(k+1)$-th epoch.
\end{proposition} 
\noindent For $k\geq 16 \alpha\kappa^2-1$, we have $n\ee_{k+1} \leq \frac{1}{8}\frac{\mm}{L^2}$.
Using this bound, it is straightforward to check that  \eqref{perbd:2} can be simplified into:
\begin{align}
     \ab_{k+1}  \leq \exp  \left( - n\ee_{k+1}\mm/2   \right) \ab_k^2+15 n^3 \ee_{k+1}^4 L^2 G^2   + 32n\ee_{k+1}^3 \kappa L G^2\,. \label{new:perbd}
\end{align}
For $k< 8\alpha \kappa^2-1$, recursively applying Proposition~\ref{per:0}  with the fact $(n\ee_{k+1})^{-1}\leq 4\kappa L +L/(2n)$ implies:
\begin{align}
    \ab_{k+1}&\leq \exp\left( - n\ee_{k+1}\mm/2 \right)\cdot \ab_k+ 3n^3\ee_{k+1}^4G^2 (4\kappa L +L/(2n))^2  + 4n\ee_{k+1}^3 \kappa L
    G^2\,.  \label{new:perbd2}
 \end{align}
 Therefore, combining \eqref{new:perbd} and \eqref{new:perbd2}, we obtain the following bound which holds for any $k\geq 1$:
 \begin{align}
    \ab_{k+1}&\leq \exp\left( - n\ee_{k+1}\mm/2 \right)\cdot\ab_k + b_2\cdot n^3\ee_{k+1}^4 +b_3\cdot n\ee_{k+1}^3 \,,  \label{new:combine}
 \end{align}
 where $b_2:= 15 L^2 G^2 +3G^2  (4\kappa L +L/(2n))^2$ and $b_3:= 32 \kappa LG^2 $, i.e., $b_2=O(\kappa^2)$ and $b_3= \OO{\kappa}$.
Following Section~\ref{sec:tight}, one can similarly modify the coefficient of $\ab_k$ in \eqref{new:combine} to obtain  
\begin{align}
     \ab_{k+1} &\leq  \exp\left(-\alpha \cdot\sum_{i=1}^n  \frac{1}{\ini+nk+i}+ \frac{\alpha }{k^2}\right)\cdot \ab_k + \frac{16 b_2 \alpha^4 n^3\mu^{-4}}{(\ini + n(k+1))^4} +\frac{8b_3 \alpha^3 n\mu^{-3}}{(\ini + n(k+1))^3} \label{new:kk} 
 \end{align}
 However, one can notice that \eqref{new:kk} is not quite of the form \eqref{kk}, and Lemma~\ref{main:2} is not directly applicable to this bound.
 In fact, we need to make some modifications in Lemma~\ref{main:2}.
First, for $\cbt>0$ and $\gamma>0$, there is an additional term to the recursive relations \eqref{kk}: for any $k\geq1$, the new recursive relations now read
\begin{align} \label{kk:new}
    \ab_{k+1} &\leq  \exp \left( -\alpha  \sum_{i=1}^{n } \frac{1}{\ini+nk+i} + \eps\cdot  \frac{1}{k^2}\right)  \ab_k +   \frac{\cbo}{(\ini+n(k+1))^{\beta+1}}+\frac{\cbt}{(\ini+n(k+1))^{\gamma+1}}\,. 
\end{align}
It turns out that for these additional terms in the recursive relations, one can use the same techniques to prove that  the global convergence bound \eqref{bound:2} has the following additional terms:
\begin{align}\label{bd:add}
       \frac{\frac{c}{\alpha-\beta}   e^{\frac{\alpha}{\ini+n+1}}\cdot \cbt }{n(\ini+nK)^\gamma} +  \frac{c  e^{\frac{\alpha}{\ini+n+1}}\cdot \cbt }{(\ini+nK)^{\gamma+1}}\,.
\end{align}
 Now using this modified version of Lemma~\ref{main:2},  one completes the proof.

	\end{document}